\documentclass[10pt,a4paper,twoside]{amsart}
\usepackage[utf8]{inputenc}
\usepackage{hyperref}
\usepackage{amsmath}
\usepackage{enumerate}
\usepackage{amsthm}
\usepackage{amssymb}
\usepackage{xcolor}
\usepackage{bm}
\usepackage{graphicx}
\author{J. Aponte}
\address{ESPOL Polytechnic University\\
Escuela Superior Politécnica del Litoral, ESPOL\\
Facultad de Ciencias Naturales y Matemáticas\\
Campus Gustavo Galindo, Km. 30.5 Vía Perimetral\\
P.O. Box 09-01-5863, Guayaquil, Ecuador}
\email{japonte@espol.edu.ec}

\author{D. Carrasco-Olivera}
\address{Departamento de Matem\'atica, Facultad de Ciencias, Universidad
del B\'{i}o-B\'{i}o, Casilla 5-C, Concepc\'{i}on. VIII-region, Chile.}
\email{dcarrasc@ubiobio.cl}

\author{H. Villavicencio}
\address{Instituto de Matem\'atica y Ciencias Afines\\ Universidad Nacional de Ingenier\'ia, Lima, Per\'u.}
\email{hvillavicencio@imca.edu.pe}


\date{\today}

\DeclareMathOperator{\Per}{Per}
\DeclareMathOperator{\Fix}{Fix}
\DeclareMathOperator{\diam}{diam}
\DeclareMathOperator{\Int}{Int}
\DeclareMathOperator{\supp}{supp}
\theoremstyle{definition}
\newtheorem{definition}{Definition}[section]
\newtheorem{example}[definition]{Example}
\theoremstyle{plain}
\newtheorem{theorem}[definition]{Theorem}
\newtheorem{lemma}[definition]{Lemma}

\newtheorem{corollary}[definition]{Corollary}
\theoremstyle{remark}
\newtheorem{remark}[definition]{Remark}
\title{Inner-distal Homeomorphisms and Measures}
\begin{document}

\begin{abstract}
    An inner-distal homeomorphism is one such that each of its proximal cells has empty interior. In locally connected spaces, we prove  these homeomorphisms have the following properties: Every $cw$-distal homeomorphism is inner-distal but not conversely.  The inner-distal homeomorphisms are precisely those for which the diameters of the iterates of every connected subset of the phase space remain uniformly bounded away from zero. If, in addition, the phase space is compact, every distal extension of an inner-distal-homeomorphism under an inner-light map is inner-distal. We also study inner-distal homeomorphisms from the measure-theoretical standpoint through the concept of inner-distal measures. We prove that in Polish metric spaces, a homeomorphism is inner-distal if and only if each Borel probability measure is. In compact metric spaces, the existence of inner-distal measures guarantees the existence of an invariant inner-distal one. In compact, locally connected spaces, the minimal homeomorphisms supporting inner-distal homeomorphism are precisely the inner-distal ones. We also classify the inner-distal homeomorphisms of the circle or the interval. Finally, we give some dynamic consequences of inner-distality.
\end{abstract}
	\maketitle
 
	\section{Introduction}
	\noindent Let $f\colon X\to X$ be a homeomorphism of a metric space $(X,d)$. We say that $f$ is \emph{distal} if
	\[
	\forall\,x,y\in X,\ \inf_{n\in \mathbb{Z}} d(f^n(x),f^n(y))=0\implies x=y.
	\]
	Equivalently, $f$ is distal if $\forall\,x\in X,\ \mathcal{P}(x)=\{x\}$, where 
	\[
	\mathcal{P}(x):=\{y\in X\colon \inf_{n\in\mathbb{Z}}d(f^n(x),f^n(y))=0\}
	\]
	is \emph{the proximal cell} of $x$ \cite{auslander1988} (sometimes written $\mathcal{P}_f(x)$ to indicate dependence on $f$).  This concept was introduced by Hilbert, who apparently sought a topological equivalent of the concept of a rigid group of motions of space. It is the weaker condition one could impose on the group of homeomorphisms $\{f^n\}_{n\in \mathbb{Z}}$ of $X$ so that this group behaves similarly to a group of rigid motions (see \cite{zippin}). The word ``distal" itself was coined by W. H. Gottschalk (see \cite{gottschalk1955}). Distal homeomorphisms have been studied extensively in the literature. For instance, R. Ellis showed that, under certain hypotheses on $X$, being distal is equivalent to saying that the topological closure of $\{f^n\}_{\in\mathbb{Z}}$ in $X^X$  has a group structure (see \cite{ellis1958}). Another quite remarkable result concerning distal homeomorphisms is the so-called Furstenberg Structure Theorem \cite{furstenberg1963}, which asserts that every minimal distal homeomorphism can be obtained by a (possibly transfinite) sequence of isometric extensions starting with the one-point dynamical system. Albeit its generality, distal homeomorphisms also have very nice recurrence properties. It is known that every distal homeomorphism is \emph{pointwise almost periodic}, that is, for every $x\in X$ and every $\varepsilon>0$, the set $\{n\in\mathbb{Z}\colon d(f^n(x),x)<\varepsilon\}$ does not have arbitrarily large gaps. Distal homeomorphisms also have very tidy dynamics, as they are known to have zero topological entropy \cite{parry1967}.
 
 In recent years, various generalizations of the notion of distality have emerged by varying the structure of the proximal cells. For example, by asking that the cardinal of the proximal cells  be less than or equal to a natural number $N$, we obtain the $N$-distal homeomorphisms \cite{AC-OLM2020,rego_salcedo_2020}. If we ask the proximal cells to be countable, we obtain countably distal homeomorphisms \cite{LM2017}. If we go further and ask that the proximal cells be totally disconnected, we will obtain the $cw$-distal homeomorphisms  \cite{AC-OLM2020}. On the other side of the dynamic behavior spectrum are expansive homeomorphisms introduced by Utz in \cite{utz1950}. We say that $f\colon X\to X$ is \emph{expansive} if there is some $\delta>0$, called the \emph{expansivity constant of $f$}, satisfying 
	\[
	\forall\,x,y\in X,\forall\,n\in\mathbb{Z},\ d(f^n(x),f^n(y))\leq \delta\implies x=y.
	\]
Equivalently, $f$ is expansive if there is $\delta>0$ such that $\Gamma_\delta(x)=\{x\}$, for all $x\in X$, where
\[
\Gamma_\delta(x):=\{y\in X\colon \forall\,n\in\mathbb{Z},\ d(f^n(x),f^n(y))\leq \delta\}.
\]
The set $\Gamma_\delta(x)$ is called the \emph{dynamic ball of radius $\delta$ and center at $x$}. The authors in \cite{bautista2018} presented a possible link between descriptive set theory and expansive systems. They defined \emph{meagre-expansive homeomorphisms} as those homeomorphisms $f\colon X\to X$ for which there exists a constant $\delta>0$, called the \emph{meagre-expansivity constant of $f$}, such that $\forall\,x\in X,\ \Int \Gamma_\delta(x)=\emptyset$, where $\Int(\cdot)$ denotes the interior operation. This paper pursues a similar link between descriptive set theory and distal systems. This link will be the concept of \emph{inner-distal homeomorphism} and \emph{inner-distal measure}.

 This paper is organized as follows. In Section 2, we present several examples and the main theorems. In Section 3, we present some preliminary results which are interesting in their own right. In Section 4, we prove our theorems.
\section{Examples and Statements of Results}
\noindent We now state the first main definition of this work. 
	\begin{definition}
		A homeomorphism $f\colon X\to X$ is said to be \emph{inner-distal} if 
  \[
 \forall\,x\in X,\ \Int\mathcal{P}(x)=\emptyset.
  \]
	\end{definition}
The following example illustrates this definition. Remind that a metric space is \emph{self-dense} if it does not have isolated points.
\begin{example}
Every distal homeomorphism on a self-dense metric space is inner-distal. This is false if the space is not self-dense. In fact, if $X=Y\cup\{p\}$, where $p$ is an isolated point and $f\colon  X\to X$ is a homeomorphism such that $f(p)=p$ and $f|_Y\colon Y\to Y$ is distal, then $f$ is distal, but it is not inner-distal since $\Int\mathcal{P}(p)=\Int \{p\}=\{p\}\neq \emptyset$. On the other hand, every metric space supporting inner-distal homeomorphisms is self-dense.
\end{example}

Now, we show an example of an inner-distal homeomorphism that is not distal. 
\begin{example}\label{example_1.3}
	Let $X=[0,1]\times C$, where $C$ is the ternary Cantor set of $[0,1]$. Let $f\colon X\to X$ be the map defined in the proof of example 1.7 in \cite{AC-OLM2020}, which is known to be not distal. Then, $\mathcal{P}(x,y)\subseteq \{x\}\times C$, for every $(x,y)\in X$. Therefore, $\Int \mathcal{P}(x,y)\subseteq \Int  (\{x\}\times C)=\emptyset$. Then, $f$ is inner-distal.
\end{example}
\begin{example}
More generally, if $f\colon X\to X$ or $g\colon Y\to Y$ is an inner-distal homeomorphism, then the product $f\times g\colon X\times Y\to X\times Y$ is also inner-distal. Indeed, it is easy to see that $\mathcal{P}_{f\times g}(x,y)\subseteq \mathcal{P}_{f}(x)\times \mathcal{P}_g(y)$, for every $(x,y)\in X\times Y$. Therefore, if $f$, say, is inner-distal, then $\Int(\mathcal{P}(x)\times \mathcal{P}(y))=\Int\mathcal{P}_f(x)\times \Int\mathcal{P}_g(y)=\emptyset\times\Int\mathcal{P}_g(y)=\emptyset$, so $\Int\mathcal{P}_{f\times g}(x,y)=\emptyset$, for every $(x,y)\in X\times Y$ and $f\times g$ is inner-distal.
\end{example}
The homeomorphism of Example \ref{example_1.3} is also an example of $cw$-\emph{distal} homeomorphisms \cite{AC-OLM2020}, i.e., homeomorphisms $f\colon X\to X$ satisfying the implication
\[
\inf_{n\in\mathbb{Z}}\diam f^{n}(C)=0\implies C\ \text{reduces to a singleton},
\] 
for every non-empty compact connected $C\subseteq X$. Theorem 1.3 in \cite{AC-OLM2020} says that the $cw$-distal homeomorphisms of a compact metric space are precisely those for which the proximal cell of each point is a totally disconnected set. It is natural to ask if every $cw$-\emph{distal} homeomorphism defined on a self-dense space is inner-distal, but this is not true, as the following example shows. Denote by $\mathcal{O}(x)$ the set $\{f^n(x)\colon n\in\mathbb{Z}\}$, the so-called \emph{orbit} of $x$. Denote by $\overline{(\cdot)}$ the closure operation.

\begin{example}
Let $C$ be the classical ternary Cantor set and let $f\colon C\to C$ be a homeomorphism such that $f(0)=0$, $f(1)=1$, and $\overline{\mathcal{O}(x)}=C$ for all $x\in C\setminus\{0,1\}$ (see \cite{gutek1984}). Clearly, $f$ is $cw$-distal, but  is not inner-distal since $\Int \mathcal{P}(0)=C\setminus\{1\}$.
\end{example}

However, in self-dense locally connected metric spaces, $cw$-distal homeomorphisms are inner-distal homeomorphisms.
\begin{example}\label{cwimpliesmeagre}
	Every $cw$-distal homeomorphism $f$ defined on a self-dense locally connected metric space $X$ is inner-distal. Indeed, suppose that for some $x\in X$, $\Int\mathcal{P}(x)\neq \emptyset$. Pick some $y\in\Int\mathcal{P}(x)$ and let $U$ be an open connected neighborhood  of $y$ contained in $\mathcal{P}(x)$. Since $\mathcal{P}(x)$ is totally disconnected, it follows that $U=\{y\}$, but this would imply that $\{y\}$ is an isolated point of $X$, an absurd. Therefore, $f$ is inner-distal.
\end{example}

The converse of this example is not valid in virtue of the following example, which shows that there are inner-distal homeomorphisms defined on locally connected compact metric spaces that are not $cw$-distal.
\begin{example}\label{inner-distal}
	Let $X=[0,1]\times \mathbb{S}^1$, where $\mathbb{S}^1$ denotes the unit circle, and let $f\colon X\to X$ be the map given by
	\[
	f(x,y)=(g(x),h(y)),
	\]
	where $g(x)=\sqrt{x}$ and $h\colon \mathbb{S}^1\to\mathbb{S}^1$ is any distal map (for instance, a rotation). Since for all $x\in (0,1)$, $g^{n}(x)\to 1$, $n\to \infty$, and $g^{n}(x)\to 0$, $n\to -\infty$, it follows that 
	\[
	\mathcal{P}(x,y)={[0,1]\times \{y\}},
	\]
 for $(x,y)\in (0,1)\times S^{1}$. Since $g(0)=0$ and $g(1)=1$, we also have
 \[
\mathcal{P}(0,y)=[0,1)\times \{y\}\ \text{and }\mathcal{P}(1,y)=(0,1]\times \{y\}
 \]
	Hence, $\forall\,(x,y)\in [0,1]\times\mathbb{S}^1,\ \Int\mathcal{P}_f(x,y)=\emptyset$. Therefore, $f$ is inner-distal, but it is not $cw$-distal.
\end{example}
It is known that an expansive homeomorphism of an infinite compact metric
space cannot be distal (e.g., \cite{utz1950}). In \cite{AC-OLM2020}, the authors proved that a $cw$-\emph{expansive} homeomorphism $f\colon X\to X$, i.e., a homeomorphism satisfying the implication
\[
\sup_{n\in\mathbb{Z}}\diam f^{n}(C)=0\implies C\ \text{reduces to a singleton},
\] 
for every non-empty compact connected $C\subseteq X$, cannot be $cw$-distal either, as long as $X$ has positive topological dimension. It is natural to believe that a meagre-expansive homeomorphism defined on an infinite compact metric space cannot be inner-distal. However, this is not true, even in compact metric spaces of positive topological dimension, as per the following example.
\begin{example}
Let $\mathbb{T}^2$ denote the 2-torus and let $g\colon\mathbb{T}^2\to \mathbb{T}^2$ be any expansive homeomorphism (e.g., an Anosov diffeomorphism). Consider the map $\mathrm{Id}_{\mathbb{S}^1}\times g\colon \mathbb{S}^1\times \mathbb{T}^2\to\mathbb{S}^1\times \mathbb{T}^2$. It is known that this map is meagre-expansive (see Proof of Theorem 1.1 in \cite{bautista2018}). This map is also inner-distal, because for each $(x,y)\in \mathbb{S}^1\times \mathbb{T}^2$ one has $\mathcal{P}(x,y)\subseteq \mathcal{P}_{\mathrm{Id}_{\mathbb{S}^1}}(x)\times \mathcal{P}_g(y)=\{x\}\times \mathcal{P}_g(y)$, whence $\Int\mathcal{P}(x,y)=\emptyset$.
\end{example}
Our first result gives a dynamic characterization of inner-distal homeomorphisms on self-dense locally connected metric spaces in the spirit of Theorem 1.3 of \cite{AC-OLM2020}. 
\begin{theorem}\label{charac_meagre}
	Let $X$ be  a self-dense  metric space and let $f\colon X\to X$. A necessary condition for $f$  to be inner-distal is
	\[ \inf_{n\in\mathbb{Z}}\diam f^n{(C)}=0\implies \Int C=\emptyset,
\]
for every connected subset $C$ of $X$. If, in addition, $X$ is locally connected, this condition is also sufficient.
\end{theorem}

For our next example, we first recall some definitions. A subset $C\subseteq X$ is \emph{nowhere-dense} if $\Int \overline{C}=\emptyset$. A \emph{continuum} is a non-empty compact connected space $X$. A \emph{subcontinuum} is a subset $C\subseteq X$, which is a continuum with respect to the subspace topology. 
\begin{example}
The inner-distal homeomorphisms in $\mathbb{S}^1$ and $[0,1]$ are precisely the $cw$-distal  ones. In effect, if $f$ is an inner-distal homeomorphism on any of these compact locally connected spaces, then for every subcontinuum  $C$ such that $\inf_{n\in \mathbb{Z}}\diam f^{n}(C)=0$, one has that $C$ is nowhere-dense, by Theorem \ref{charac_meagre}; still the only non-empty compact connected nowhere-dense  subsets of $\mathbb{S}^1$ and $[0,1]$ are the singletons. Therefore, $f$ is $cw$-distal.
\end{example}

A \emph{homomorphism} from the homeomorphism $g\colon Y\to Y$ to the homeomorphism $f\colon X\to X$ is a continuous surjective map $\pi\colon Y\to X$ satisfying $f\circ\pi=\pi\circ g$. In this case, we say that $g$ is an \emph{extension} of $f$ and that $f$ is a \emph{factor} of $g$. We say that $\pi\colon Y\to X$ is \emph{light} if $\pi^{-1}(x)$ is totally disconnected,  for every $x\in X$. These definitions were used in \cite{AC-OLM2020}, and it was proved that every extension of a $cw$-distal homeomorphism under a light map is $cw$-distal. 

To obtain an inner-distal version of this, observe that $\pi$ is light if and only if for every subcontinuum $C$ in $Y$ such that $h(C)$ reduces to a point, we have that $C$ reduces to a point. Motivated by this observation, we will say that $\pi \colon Y\to X$ is \emph{inner-light} if $\Int C=\emptyset$, for every  subcontinuum $C$ such that $\Int \pi(C)=\emptyset$. With this definition, we have the following result.

\begin{theorem}\label{inner-distal-extensions}
In self-dense  locally connected compact metric spaces, an inner-distal extension of an inner-distal homeomorphism is inner-distal.
\end{theorem}
Notice that the homeomorphism defined in Example \ref{inner-distal} is indeed an  extension of a distal (hence inner-distal) homeomorphism of the unit circle under an inner-light map. Since the involved spaces are self-dense, compact, and locally-connected, it follows from Theorem \ref{inner-distal-extensions} that the aforementioned homeomorphism is inner-distal.

Next, we analyze inner-distality from the measure-theoretical standpoint. We state our second main definition.
\begin{definition}
	A Borel probability measure $\mu$ of a metric space $X$ is \emph{inner-distal} with respect to a homeomorphism $f\colon X\to X$ (or simply inner-distal) if \[\forall\,x\in X,\ \mu(\Int \mathcal{P}(x))=0.\]
\end{definition}
The following result relates the inner-distal homeomorphisms with inner-distal measures. Compare with Theorem 1.3 in \cite{LM2017}.
\begin{theorem}\label{meagre_dis_meas}
	Homeomorphisms defined on self-dense Polish metric spaces  are inner-distal if and only if every Borel probability measure is  inner-distal.
\end{theorem}
Recall that a Borel probability measure $\mu$ is \emph{invariant} under the homeomorphism $f\colon X\to X$ if $\mu(A)=\mu(f^{-1}(A))$ for every Borel subset $A$ of $X$. Our following result guarantees invariant inner-distal measures for a homeomorphism $f$ as long as at least one inner-distal measure exists.
\begin{theorem}\label{meagre-invariant}
In compact metric spaces, every homeomorphism  with inner-distal measures has invariant inner-distal measures.
\end{theorem}
As an application of Theorem \ref{meagre-invariant}, we have the following theorem that presents a strong dichotomy between the existence or not of inner-distal measures for minimal homeomorphisms.

\begin{theorem}\label{minimal_inner_distal}
In locally connected compact metric spaces, every homeomorphism supporting inner-distal measures is inner-distal.
\end{theorem}
Next, we study the homeomorphisms on the circle or the interval supporting meagre-expansive measures. We recall some definitions. Given a Borel probability measure $\mu$ on $X$, its \emph{support} is the set $\supp (\mu):=\{x\in X\colon \mu(U)>0\ \text{for every neighborhood }U\ \text{of }x\}$.
We say that $\mu$ has \emph{full support} if $\supp(X)=X$. The \emph{Pole North-South} homeomorphism on the interval is the map $f\colon [0,1]\to [0,1]$ given by $f(x)=1-x$. A Borel probability measure is \emph{meagre-expansive}, with respect to $f$, if there is $\delta>0$ (called  \emph{meagre-expansivity constant}) such that for all $x\in X$, $\mu(\Int \Gamma_\delta(x))=0$ \cite{bautista2018}. The \emph{rotation number} of a homeomorphism $f\colon\mathbb{S}^1\to \mathbb{S}^1$ is the number
\[
\lim_{n\to\infty }\frac{F^n(t)-t}{n},
\]
where $x$ is any point of $\mathbb{R}$ and $F\colon \mathbb{R}\to \mathbb{R}$ is a \emph{lift} of $f$, i.e., $F$ is a homeomorphism satisfying $f(e^{it})=e^{iF(t)}$, for each $t\in\mathbb{R}$. It is known that the rotation number of $f$ neither depends on $x$ nor the lift (see \cite{katok_hasselblatt}). A homeomorphism $f\colon X\to X$ is transitive if there is $x\in X$ such that $\overline{\mathcal{O}(x)}=X$.
A \emph{Denjoy map} is a circle homeomorphism that is not transitive but has irrational rotation number. We say $p\in X$ is a \emph{periodic point} for $f$ if there is a non-negative integer $n$ such that $f^n(p)=p$. We denote by $\Per(f)$ the set of periodic points of $f$. We say that a Borel probability measure is \emph{meagre-expansive} with respect to $f$ (or simply meagre-expansive) if there is $\delta>0$ (called \emph{measure-expansivity constant}) such that $\mu(\Int\Gamma_\delta(x))=0$, for all $x\in X$ \cite{bautista2018}.

\begin{theorem}\label{circ_I_megre_dis} The following statements  hold for  homeomorphisms on the circle or the interval:
\begin{enumerate}
\item The inner-distal measures of circle homeomorphisms with rational rotation number are  supported in the set of periodic points of the map. In addition, if the set of periodic points is finite, then every inner-distal measure is meagre-expansive.
\item Every circle homeomorphism with irrational rotation is either  distal or is topologically conjugate to a Denjoy map, and every inner-distal measure is meagre-expansive.
\item A circle homeomorphism has an inner-distal measure with full support if and only if it is distal. In particular, the only inner-distal homeomorphisms on the circle are the distal ones.
\item An interval homeomorphism has inner-distal measures with full support if and only if it is the identity or the Pole North-South homeomorphism.
\end{enumerate}
\end{theorem}

Our last results deal with some dynamic consequences of inner-distality. 
We recall some definitions. A Borel probability measure is a \emph{distal measure} for a homeomorphism $f\colon X\to X$, if $\forall\,x\in X,\ \mu(\mathcal{P}(x))=0$. A subset $A\subseteq\mathbb{Z}$ is \emph{syndetic} if there is a closed subset $F\subseteq \mathbb{Z}$ with $\mathbb{Z}=A+F$. We say that $x\in X$ is a \emph{pointwise almost periodic} point of a  homeomorphism $f\colon X\to X$, if for every open neighborhood $U$ of $x$, the set $\{n\in \mathbb{Z}\colon f^n(x)\in U\}$ is syndetic \cite{auslander1988}.  Let $AP(f)$ denote the set of almost pointwise periodic points of $f$. A \emph{stable class} of a homeomorphism $f\colon X\to X$ is a subset equal to $W^s(x):=\{y\in X\colon \lim_{n\to\infty }d(f^n(x),f^{n}(y))=0\}$, for some $x\in X$. It is known that if $f$ has a distal measure, then $AP(f)$ is uncountable, and $f$ has uncountably many stable classes \cite{LM2017}. 

We have an analogous result for inner-distal homeomorphisms. Recall that a subset $A$ of $X$ is \emph{meagre} if $A$ can be written as a countable union of nowhere-dense subsets of $X$.
\begin{theorem}\label{theorem_ap}
Let $f\colon X\to X$ be an inner-distal homeomorphism defined on a compact metric space $X$. If every proximal cell  is meagre, then $f$ has uncountable almost periodic points and uncountably many stable classes.
\end{theorem}

The following example illustrates this theorem.
\begin{example}
Let $f\colon \mathbb{S}^1\times \mathbb{S}^1\to\mathbb{S}^1\times \mathbb{S}^1$ be the homeomorphism given in complex notation by
\[
f(e^{2\pi i t_1},e^{2\pi i t_2})=(e^{2\pi i \sqrt{t_1}},e^{2\pi i(t_2+\alpha)}),
\]
 Then, for each $(z_1,z_2)\in \mathbb{S}^1\times\mathbb{S}^1$, $\mathcal{P}(z_1,z_2)=\mathbb{S}^1\times\{z_2\}$ which is a closed set with empty interior. So, by Theorem \ref{theorem_ap}, $AP(f)$ is uncountable, and $f$ has uncountably many stable classes. In fact, $AP(f)=\{1\}\times \mathbb{S}^1$, and the stable classes are $W^s(z_1,z_2)=\mathbb{S}^1\times \{z_2\}$.
\end{example}

It is known that if a $N$-distal homeomorphism is transitive, then either $X$ is a periodic orbit or $f$ does not have periodic points (See Proposition 4.2 in \cite{rego_salcedo_2020}). To obtain an inner-distal version, recall that $f$  is \emph{totally transitive} if $f^n$ is transitive for every $n\in\mathbb{Z}$. With this definition, we have the following result.
The following example illustrates this theorem.
 \begin{theorem}\label{totally-transitive}
The set of periodic points of an inner-distal totally transitive homeomorphism of a self-dense compact metric space has empty interior.
 \end{theorem}
 We have a related example.
\begin{example}
Let $\Sigma^2=\{0,1\}^{\mathbb{Z}}$ with the  metric $d((x_n)_n,(y_n)_n)=\frac{1}{2^m}$, where $m=\min\{|n|\colon x_n\neq y_n\}$, and let $\sigma\colon \Sigma^2\to \Sigma^2$  be the shift $\sigma(x_n)_n=(x_{n+1})_{n}$. It is well known that $\sigma$ is totally transitive (see \cite{bhaumik2011}) and that $\overline{\Per(\sigma)}=\Sigma^2$. It is an easy exercise to see that $\sigma$ is also inner-distal. Hence, $\Int\Per(\sigma)=\emptyset$. 
\end{example}
\begin{remark}
   The previous example also shows that inner-distal systems with positive topological entropy exist, contrary to the case of distal systems (see \cite{parry1967}).
\end{remark}

\section{Preliminaries}

\noindent Let $\mathcal{M}(X)$ denote the set of probability measures of the metric space $X$ equipped with the weak* topology defined by the convergence $\mu_n\to \mu$ if and only if $\int\varphi\,d\mu_n\to \int\varphi\,d\mu$, for every continuous map $\varphi\colon X\to \mathbb{R}$.
Let $\mathcal{M}_{idis}(f)$ denote the set of inner-distal measures of $f$. Let $f_*\colon \mathcal{M}(X)\to \mathcal{M}(X)$ be the map defined by $f_*(\mu)(E)=\mu(f^{-1}(E))$, for every Borel probability measure $\mu$, and every Borel subset $E$ of $X$.

\begin{lemma}\label{idis_measures}
If $f\colon X\to X$  is a compact metric space homeomorphism, then $\mathcal{M}_{idis}(f)$ is an $f_*$-invariant closed convex subset of $\mathcal{M}(X)$.
\end{lemma}
\begin{proof} We follow the arguments in \cite{bautista2018}.
Let $\mu\in \mathcal{M}_{idis}(f)$ and $x\in X$. Then,
\begin{align*}
f_*(\mu)(\Int\mathcal{P}(x))&=\mu(f^{-1}(\Int \mathcal{P}(x)))\\
&=\mu(\Int f^{-1}(\mathcal{P}(x)))\\
&=\mu(\Int \mathcal{P}(f^{-1}(x)))\\
&=0.
\end{align*}
Hence, $f_*(\mu)\in\mathcal{M}_{idis}(f)$. Now let $\mu,\nu\in \mathcal{M}_{idis}(f)$, $t\in[0,1]$ and $x\in X$. We have
\begin{align*}
((1-t)\mu+t\nu)(\Int\mathcal{P}(x))&=(1-t)\mu(\Int\mathcal{P}(x))+t\nu(\Int\mathcal{P}(x))=0.
\end{align*}
Thus, $\mathcal{M}_{idis}(f)$ is convex. It remains to prove that $\mathcal{M}_{idis}(f)$ is closed. Let $(\mu_n)_{n\in\mathbb{N}}$ be a sequence in $\mathcal{M}_{idis}(f)$ converging to $\mu\in\mathcal{M}(X)$. Then for $x\in X$, $\mu_n(\Int \mathcal{P}(x))=0$, so $\liminf_n\mu_n(\Int \mathcal{P}(x))=0$. Since $\Int \mathcal{P}(x)$ is open, it follows from well-known properties of weak* convergence of measures \cite{P1967} that
\[
0=\liminf_n{\mu_n}(\Int \mathcal{P}(x))\geq \mu(\Int \mathcal{P}(x))=0.
\]
Therefore, $\mu\in \mathcal{M}_{idis}(f)$, and $\mathcal{M}_{idis}(f)$ is closed. This proves the lemma.
\end{proof}
\begin{lemma}\label{lemma_idis_meas}
	In a  metric space $X$, a necessary condition for a Borel probability measure $\mu$ to be inner-distal with respect to a homeomorphism $f$ is that for every connected set $C$ satisfying $\mu(\Int C)>0$, one has \[\inf_{n\in\mathbb{Z}}\diam f^{n}(C)>0.\]
 If, in addition, $X$ is separable and locally connected, this condition is also sufficient.
\end{lemma}
\begin{proof}
First, suppose that $\mu$ is inner-distal with respect to $f$.  If $\inf_{n\in\mathbb{Z}}\diam f^n(C)=0$, then $C\subseteq \mathcal{P}(x)$ for any $x\in C$. Hence, $\mu(\Int C)\leq \mu(\Int\mathcal{P}(x))=0$ (this implication is valid for every subset $C$ of $X$). Now suppose that for every connected $C$ with $\mu(\Int C)>0$ one has $\inf_{n\in\mathbb{Z}}\diam f^{n}(C)>0$, and let $x\in X$. Let us prove that $\mu(\Int\mathcal{P}(x))=0$. We can suppose that $\Int \mathcal{P}(x)\neq \emptyset$. Since $X$ is locally connected, for every $y\in \Int \mathcal{P}(x)$, there is an open connected $U_y$ such that $y\in U_y$ and $U_y\subseteq \Int\mathcal{P}(x)$. Therefore, $\{U_y\}_{y\in\Int \mathcal{P}(x)}$ is an open cover of $\Int\mathcal{P}(x)$. Besides, $\inf_{n\in\mathbb{Z}}\diam f^n(U_y)=0$, for $U_y$ is a connected set contained in $\mathcal{P}(x)$ (see the proof of the sufficiency part of Theorem \ref{charac_meagre}), so $\mu(U_y)=0$, by the assumption. Since $X$ is separable, $\Int\mathcal{P}(x)$ is Lindelöf, so there are points $y_1,y_2,\dots$ in $\Int\mathcal{P}(x)$ with $\Int\mathcal{P}(x)=\bigcup_{i=1}^\infty U_{y_i}$. Therefore, $\mu(\Int\mathcal{P}(x))\leq\sum_{i=1}^\infty\mu(U_{y_i})=0$. This proves the lemma. 
\end{proof}

Recall that a Borel probability measure $\mu$ is a \emph{$cw$-distal measure} for a homeomorphism $f\colon X\to X$, if for every subcontinuum $C$ with $\mu(C)>0$ one has $\inf_{n\in \mathbb{Z}}\diam f^n(C)>0$ \cite{AC-OLM2020}. The following corollary is immediate.

\begin{corollary}
In a locally connected separable metric space, every $cw$-distal measure is inner-distal.
\end{corollary}
\begin{lemma}\label{iterate-inner-distal}
For every homeomorphism $f\colon X\to X$ of a metric space $X$ and every integer $k$, every inner-distal measure $\mu$ for $f$ is also an inner-distal measure for $f^k$. In particular, if $f$ is inner-distal, then $f^k$ is inner-distal.
 \end{lemma}
 \begin{proof}
 Since $\inf_{n\in\mathbb{Z}}d(f^{kn}(x),f^{kn}(x))=0$ implies $\inf_{n\in\mathbb{Z}}d(f^{n}(x),f^{n}(x))=0$, it follows that $\mathcal{P}_{f^k}(x)\subseteq \mathcal{P}(x)$, so $\Int\mathcal{P}_{f^k}(x)\subseteq \Int\mathcal{P}(x)$ and the proof follows.
 \end{proof}
Recall that the \emph{non-wandering set} of a homeomorphism $f \colon X \to X$ is the closed set of points $x \in X$, denoted by $\Omega(f)$, satisfying
$U\cap \left(\bigcup_{n=1}^\infty f^n(U)\right)\neq \emptyset$, for every neighborhood $U$ of $x$. 

\begin{lemma}\label{supp_circle}
Let $X$ denote  the circle $\mathbb{S}^1$ or the unit interval $[0,1]$. If $\mu$ is an inner-distal measure with respect to a homeomorphism $f\colon X\to X$, then $\supp(\mu)\subseteq \Omega(f)$.
\end{lemma}
\begin{proof}
	Suppose on the contrary that for some non-atomic measure $\supp(\mu)\subsetneq \Omega(f).$ Since $\Omega(f)$ is closed, $X\setminus\Omega(f)$ is a disjoint collection of (relatively) open intervals $I$. Note that $\lim_{n\to\pm\infty}\diam f^{n}(I)=0$. On the other hand, since $\supp(\mu)\subsetneq\Omega(f)$, then $\mu(I)>0$ for some of the aforementioned intervals $I$. However, $\mu$ is inner-distal, and $X$  is separable and locally connected, so Lemma \ref{lemma_idis_meas} implies that $\mu(I)=0$, an absurd. This contradiction proves the lemma.
\end{proof}
For the next lemma, we recall some definitions. A homeomorphism $f\colon X\to X$ is said to be \emph{asymptotically expansive} if there is $\delta>0$ (called \emph{asymptotic expansivity constant}) such that if $x,y\in X$ and $d(f^n(x),f^n(y))\leq \delta$ for every integer $n\geq 0$, then
\[
\lim_{n\to\infty}d(f^n(x),f^n(y))=0.
\]
We say that $f$ is \emph{bi-asymptotically expansive} if $f$ and $f^{-1}$ are asymptotically expansive \cite{lee_morales_villa_2022}. Naturally, we can also say that $\mu$ is \emph{positively meagre-expansive} if there is $\delta>0$ (called \emph{positively measure-expansivity constant}) such that, for all $x\in X$,
\[
\mu(\Int \Gamma_\delta^+(x))=0,
\]
where 
\[
\Gamma^+_\delta (x):=\{y\in X\colon \forall\,n\geq0,\ d(f^n(x),f^n(y))\leq \delta\}
\]
is the \emph{positively dynamic ball of radius $\delta$ and center at $x$}.

\begin{lemma}\label{asymp-exp-idis}
Every inner-distal measure of an asymptotically \sloppy (resp. bi-asymptotically) expansive homeomorphism is positively meagre-expansive (resp. meagre-expansive).
\end{lemma}

\begin{proof}
We prove the lemma only for inner-distal measures of an asymptotically expansive homeomorphism. Let $\mu$ be an inner-distal measure of an asymptotically expansive homeomorphism $f\colon X\to X$ and suppose that $\delta>0$ is an asymptotic expansivity constant for $f$. Given $x\in X$, if $y\in \Gamma^+_\delta (x)$, it follows that $d(f^n(x),f^n(y))\leq \delta$ for every integer $n\geq 0$, so from the asymptotic expansivity it holds that $\lim_{n\to\infty}d(f^n(x),f^n(y))=0$, in particular, $y \in \mathcal{P}(x)$. Thus, $\Gamma^+_\delta (x)\subset \mathcal{P}(x)$. By the inner-distality of $\mu$, it follows $\mu (\Int(\Gamma^+_\delta (x)))\leq \mu(\Int\mathcal{P}(x))=0$; hence $\mu$ is positively meagre-expansive. 
\end{proof}
	We recall some definitions from Measurable Baire Theory (see Section 3 in \cite{bautista2018}). Given a Borel probability measure $\mu$ of $X$, we say that $F\subseteq X$ is \emph{$\mu$-nowhere-dense} if $\mu(\Int(\overline{F}))=0$. A subset of $X$ is \emph{$\mu$-meagre} if it is the union of countably many $\mu$-nowhere-dense subsets of $X$. Notice that the union of countably many $\mu$-meagre subsets is again $\mu$-meagre as is every  subset of a $\mu$-meagre set. We say that $X$ is $\mu$-Baire if $\mu(\Int(A))=\emptyset$ for every $\mu$-meagre subset $A$ of $X$. In particular, if $X$ is $\mu$-Baire, $X$ cannot be written as a countable union of $\mu$-nowhere-dense subsets. 
\begin{lemma}\label{lemma_ap_stable}
Let $f\colon X\to X$ be a homeomorphism defined on a complete metric space $X$ and let $\mu$ be an inner-distal Borel probability measure on $X$. If $X$ is $\mu$-Baire, and every proximal cell  is $\mu$-meagre, then $f$ has uncountably many almost periodic points and uncountably many stable classes.
\end{lemma}
\begin{proof}
We first prove that $AP(f)$ is uncountable. It is known that for each $x\in X$, there exists $x^*\in AP(f)$ such that $x^*\in \mathcal{P}(x)$, by Theorem 3, p. 67 in \cite{auslander1988}. So, $x\in \mathcal{P}(x^*)$. Therefore, $X=\bigcup_{x^*\in AP(f)}\mathcal{P}(x^*)$. Consequently, $AP(f)$ cannot be countable, for it would imply that  $X$ is a countable union of $\mu$-nowhere-dense sets, contradicting that $\mu$ is $\mu$-Baire. Now we prove that $f$ has uncountably many stable classes. First, notice that, for each $x\in X$, $W^s(x)\subseteq P(x)$, so $W^s(x)$ is $\mu$-nowhere-dense. It is clear that  the family of stable classes $\{W^s(x)\}_{x\in X}$ forms a partition of $X$. If this family were countable, then there would be points $x_1,x_2,\dots$ in $X$ such that $X=\bigcup_{x_i\in X}W^s(x_i)$, but this is, again, impossible. Therefore, there cannot be countably many stable classes either.
\end{proof}
\begin{lemma}\label{totally_transitive_measure}
Let $f\colon X\to X$ be a totally transitive homeomorphism supporting an inner-distal measure $\mu$. If $X$ is $\mu$-Baire, then the interior of the set of periodic points of $f$ has measure zero with respect to $\mu$.
\end{lemma}
\begin{proof}
Denote by $\Fix(f)$ the of fixed points of $f$, that is, the set of points $p$ such that $f(p)=p$. Clearly, $\Fix(f)$ is a closed subset of $X$. We claim that $\Fix(f)$ is $\mu$-nowhere-dense. Indeed, since $f$ is transitive, there exists an $x$ such that $\overline{\mathcal{O}(x)}=X$. This implies that $\Fix(f)\subseteq \mathcal{P}_f(x)$, so $\mu(\Int \Fix(f))=0$. Consequently, for every $n\in\mathbb{Z}$, $\mu(\Fix(f^n))=0$, too, for $f^n$ is transitive and $\mu$ is also inner-distal with respect to $f^n$, by Lemma \ref{iterate-inner-distal}. On the other hand, it is clear that
\[
\Per(f)=\bigcup_{n=1}^\infty \Fix(f^n).
\]
Since $X$ is $\mu$-Baire, it follows that $\mu(\Int\Per(f))=0$.
\end{proof}
\section{Proofs of the Theorems}
\begin{proof}[Proof of Theorem \ref{charac_meagre}]
	First, suppose that  $f$ is inner-distal. Let $C$ be a subset of $X$ such that $\inf_{n\in\mathbb{Z}}\diam f^n{(C)}=0$. Take some $x\in C$. For each $y\in C$, $d(f^n{(x)},f^{n}(y))\leq \diam f^n(C)$, so $\inf d(f^{n}(x),f^n(y))=0$. This implies $C\subseteq \mathcal{P}(x)$. In consequence,
	\[
	\Int C\subseteq \Int \mathcal{P}(x)=\emptyset,
	\]
that is, $C$ has empty interior (notice that we did not require $C$ to be connected).  

Conversely, suppose that, for all connected subsets $C$ of $X$ such that \sloppy $\inf_{n\in\mathbb{Z}}\diam f^n{(C)}=0$, one has that $C$ has empty interior. Let us see that if $C$ is a connected set contained in $\mathcal{P}(x)$, for some $x\in X$, one has that $\Int C=\emptyset$. Indeed, as observed in \cite{AC-OLM2020}, the definition  of $\mathcal{P}(x)$ implies that there is a map $N\colon C\to \mathbb{N}$ such that, for all $y\in C$,
\[
N(y)=\sup\{n\in\mathbb{N}\colon d(f^i(x),f^i(y))>\varepsilon,\ \text{for all }-n\leq i\leq n\}.
\]
It is clear that $N$ is continuous, so $N(y)=N$ is constant since $C$ is connected. It follows that $\diam f^{N+1}(C)\leq \varepsilon$, showing that
\[
\inf_{n\in\mathbb{Z}}\diam f^n(C)=0.
\]
Since $f$ is inner-distal, it follows $\Int C=\emptyset$. Suppose that for some $x\in X$, $\Int \mathcal{P}(x)\neq \emptyset$, and let $y\in \Int \mathcal{P}(x)$. Since $X$ is locally connected,  an open connected neighborhood $U$ of $y$ exists with $U\subseteq \Int\mathcal{P}(x)$, a contradiction. This proves the theorem.
\end{proof}

\begin{proof}[Proof of Theorem \ref{inner-distal-extensions}] Let $X$ and $Y$ be self-dense compact locally connected metric spaces and let $g\colon Y\to Y$ be an extension of the inner-distal homeomorphism $f\colon X\to X$, under an inner-light homomorphism $\pi\colon Y\to X$. We are going to prove that $g$ is inner-distal. Let $C$ be a connected set such that $\inf_{n\in\mathbb{Z}}\diam g^n(C)=0$. In particular, $\inf_{n\in\mathbb{Z}}\diam g^n(\overline{C})=0$. Then there is a sequence $(n_k)_{k\in\mathbb{N}}$ of integers such that $(g^{n_k}(\overline{C}))_{k\in\mathbb{N}}$ converges to some point $y\in Y$, with respect to the Hausdorff metric induced on the set of subcontinua of $Y$. Since $\pi$ is a homomorphism, it follows that $\pi(g^{n_k}(\overline{C}))=f^{n_k}(\pi(\overline{C}))$, so $f^{n_k}(\pi(\overline{C}))_{k\in\mathbb{N}}$ converges to $\pi(y)$,  with respect to the Hausdorff metric induced on the set of subcontinua of $X$. This implies that $\Int \pi(\overline{C})=\emptyset$, for $f$ is inner-distal and $X$ is self-dense and locally connected. Since $\pi$ is inner-light, we conclude that $\Int \overline{C}=\emptyset$, and so $\Int C=\emptyset$. Since $Y$ is self-dense and locally connected,  $g$ is inner-distal, as wanted.
\end{proof}

\begin{proof}[Proof of Theorem \ref{meagre_dis_meas}]
	It is clear that if $f$ is an inner-distal homeomorphism, then every Borel probability measure $\mu$ satisfies $\mu(\Int \mathcal{P}(x))=0$. Now, suppose that every Borel probability measure on $X$ is inner-distal. Let us see that $f$ is inner-distal. If $f$ is not inner-distal, then there exists  some $x\in X$ such that $\Int \mathcal{P}(x)\neq \emptyset$. Observe that $\Int\mathcal{P}(x)$ cannot be countable; otherwise, $\Int \mathcal{P}(x)=\bigcup_{i=1}^\infty\{x_i\}$. Since $X$ is self-dense, each $\{x_i\}$ is closed and has empty interior, so by Baire Category Theorem, $\Int \mathcal{P}(x)=\emptyset$, contrary to our assumption. Hence, $\Int \mathcal{P}(x)$ is uncountable. By Theorem 2.8 on p. 12 of \cite{P1967}, there is a Cantor set $C\subseteq \Int \mathcal{P}(x)$. By theorem 9.1 on p. 53 of \cite{P1967}, there is a non-atomic probability measure $\mu$ with $\supp(\mu)\subseteq C$. However, $\mu(\Int \mathcal{P}(x))=0$, whence $\mu(C)=0$, an absurd. Hence, $f$ is inner-distal. 
\end{proof}
\begin{proof}[Proof of Theorem \ref{meagre-invariant}] It follows from Lemma \ref{idis_measures} and from the general fact that every non-empty closed convex $f_*$-invariant subset of $\mathcal{M}(X)$ contains invariant measures for $f$ (see Lemma 2.4 in \cite{bautista2018}).
\end{proof}
\begin{proof}[Proof of Theorem 
\ref{minimal_inner_distal}]
Let $X$ be a locally connected compact metric space and  $f\colon X\to X$ be a minimal homeomorphism. If $f$  supports an inner-distal measure, then it supports an invariant inner-distal measure $\mu$, by Lemma \ref{idis_measures}. It is an easy exercise to verify that, on compact metric spaces, if $f$ is minimal, then $\mu$ is fully supported. In particular,  $\mu(A)>0$, for every non-empty open subset $A$ of $X$. Let $C$ be a connected subset of $X$ with $\Int C\neq \emptyset$. Then $\mu(\Int C)>0$. Since $\mu$ is inner-distal and $X$ is a compact  metric space, it follows from Lemma \ref{lemma_idis_meas} that $\inf_{n\in\mathbb{Z}}\diam f^n(C)>0$. Since $X$ is also locally connected, $f$ is inner-distal. 
\end{proof}
\begin{proof}[Proof of Theorem \ref{circ_I_megre_dis}]
Let $f\colon \mathbb{S}^1\to\mathbb{S}^1$ be a homeomorphism. We first prove (1).  If $f$ has rational rotation number and $\mu$ is an inner-distal measure for $f$, then $\supp(\mu)\subseteq\Omega(f)=\Per(f)$, by Lemma \ref{supp_circle}. If, in addition, $\Per(f)$ is finite, then $f$ is bi-asymptotically expansive (Theorem 8 in \cite{lee_morales_villa_2022}), so by Lemma \ref{asymp-exp-idis} every inner-distal measure for $f$ is meagre-expansive.
 We now prove (2). If $f$ has irrational rotation number, then  either $f$ is  topologically conjugate to an irrational rotation, in which case it is distal, or $f$ is topologically conjugate to a Denjoy map. But a Denjoy map is bi-asymptotically expansive (Theorem 8 in \cite{lee_morales_villa_2022}), so by Lemma 31 in \cite{lee_morales_villa_2022}, $f$ is bi-asymptotically expansive too. In this case, Lemma \ref{asymp-exp-idis} implies that every inner-distal measure for $f$ is meagre-expansive. Finally, we prove (3) and (4). Let $f\colon X\to X$ be a homeomorphism  of the circle or the unit interval. First, suppose that $f$ has an inner-distal measure with full support. By Lemma \ref*{supp_circle}, $X=\Omega(f)$. If $X$ is the circle, then $f$ is topologically conjugate to a circle rotation, in which case it is distal. On the other hand, if $X$ is the interval, then either $f$ is the identity of $X$ or $f$ is the Pole North-South homeomorphism. Conversely, If $f$ is inner-distal, then $f$ has inner-distal measures with full support (for instance, the Lebesgue measure on $X$). Hence, $f$  is distal in the case $X$ is the circle, or either the identity or the Pole North-South homeomorphism in the case $X$ is the interval.
\end{proof}
\begin{proof}[Proof of Theorem \ref{theorem_ap}]
Since $f$ is inner-distal,  every probability measure of $X$ is inner-distal. Since $X$ is compact, a fully supported Borel probability measure exists on $X$ (Lemma 3.6 in \cite{bautista2018}). By Corollary 3.3 in \cite{bautista2018}, $X$ is $\mu$-Baire. Since every meagre set is  $\mu$-meagre, the result then follows from Lemma \ref{lemma_ap_stable}.
\end{proof}
\begin{proof}[Proof of Theorem \ref{totally-transitive}]
Suppose by the contrary that $\Int\Per(f)\neq\emptyset$. By the same argument of the sufficiency part of Proof Theorem \ref{meagre_dis_meas}, we conclude that a Cantor set $C$ is contained in $\Int\Per(f)$ and a Borel probability measure $\mu$ with $\supp(\mu)\subseteq C$. Since $X$ is compact, by Lemma 3.6 in \cite{bautista2018}, there is a sequence $(\mu_n)_{n\in\mathbb{N}}$ of fully supported Borel probability measures in $X$ with $\mu_n\to\mu$. Since $f$ is inner-distal, each $\mu_n$ is inner-distal too. Then, for each $n\in \mathbb{N}$, $X$ is $\mu_n$-Baire, in virtue of Corollary 3.3 in \cite{bautista2018}. By Lemma \ref{totally_transitive_measure}, $\mu_n(\Int\Per(f))=0,\ \forall\,n\in\mathbb{N}$. Since $\Int \Per(f)$ is open, it follows from well-known properties of weak* convergence of measures \cite{P1967} that
\[
0=\liminf_n{\mu_n}(\Int\Per(f))\geq \mu(\Int\Per(f))=1,
\]
an absurd. This contradiction proves the result.
\end{proof}

\end{document}